\documentclass[12pt, reqno]{amsart}%
\usepackage{amsmath, amsthm, amscd, amsfonts, amssymb, graphicx, color}
\usepackage[bookmarksnumbered, colorlinks, plainpages]{hyperref}
\usepackage{amsmath}
\usepackage{amsfonts}
\usepackage{amssymb}
\usepackage{graphicx}%
\providecommand{\U}[1]{\protect \rule{.1in}{.1in}}
\newtheorem{theorem}{Theorem}[section]
\newtheorem{lemma}[theorem]{Lemma}

\theoremstyle{definition}
\newtheorem{definition}[theorem]{Definition}

\theoremstyle{remark}

\numberwithin{equation}{section}
\begin{document}

\title{A Note on Inextensible Flows of Curves in $E_{1}^{n}$}
\author{\"{O}NDER G\"{O}KMEN YILDIZ}
\address{Department of Mathematics, Faculty of Sciences and Arts, Bilecik \c{S}eyh Edebali
University, Bilecik, TURKEY}
\email{ogokmen.yildiz@bilecik.edu.tr}
\author{MURAT TOSUN}
\address{Department of Mathematics, Faculty of Sciences and Arts, Sakarya University, Sakarya, TURKEY}
\email{tosun@sakarya.edu.tr}
\maketitle

\begin{abstract}
In this paper, we study inextensible flows of non-null curves in $E_{1}^{n}$.
We give necessary and sufficient conditions for inextensible flow of non-null
curve in $E_{1}^{n}$.

\end{abstract}

\section{Introduction}

Flow of curves has a very important place in the field of industry such as
modeling ship hulls, buildings, airplane wings, garments, ducts, automobile
parts. Moreover Chirikjian and Burdick describe the kinematics of
hyperredundant (or "serpentine") robot as the flow of plane curve
\cite{chirikjian}. The flow of a curve is said to be inextensible if, its
arclength is preserved. Firstly, Kwon and Park studied inextensible flows of
curves and developable surfaces, which its arclength is preserved, in
Euclidean 3-space \cite{kwon}.

Inextensible curve flows conduce to motions in which no strain energy is
induced in physical science. For example, the swinging motion of a cord of
fixed length can be represented by this type of curve flows. Also inextensible
flows of curves have great importance in computer vision and computer
animation moreover structural mechanics (see \cite{desbrun},\cite{kass},
\cite{lu}).

There are many studies in the literature on plane curve flows, especially on
evolving curves in the direction of their curvature vector field (referred to
by various names such as \textquotedblleft curve shortening\textquotedblright,
flow by curvature\textquotedblright \ and \textquotedblleft heat
flow\textquotedblright). Among them, perhaps, most important case (but already
a very subtle one) is the curve-shortening flow in the plane studied by Gage
and Hamilton \cite{gage} and Grayson \cite{grayson}. Another paper about curve
flows was studied by Chirikjian \cite{Chirikjian2}.

Inextensible flows of curves are studied in many different spaces. For
example, G\"{u}rb\"{u}z have examined inextensible flows of spacelike,
timelike and null curves in \cite{gurbuz}. After this work
\"{O}\u{g}renmi\c{s} et al. have studied inextensible curves in Galilean space
\cite{ogrenmis} and Y\i ld\i z et al. have studied inextensible flows of
curves according to darboux frame in Euclidean 3-space \cite{yildiz} and they
have investigated inextensible flows of curves in Euclidean n-space
\cite{yildiz2}, etc.

In the present paper following \cite{kwon}, \cite{gurbuz}, \cite{ogrenmis},
\cite{yildiz}, \cite{yildiz2}, we study inextensible flows of non-null curves
in $E_{1}^{n}.$ Further, we give necessary and sufficient conditions for
inextensible flows of non-null curves in $E_{1}^{n}.$

\section{Preliminaries and Notations}

\bigskip Let $E_{1}^{n\text{ }}$ be the $n-$dimensional pseudo-Euclidean space
with index 1 endowed with the indefinite inner product given by%
\[
\left \langle X,Y\right \rangle =-x_{1}y_{1}+%
%TCIMACRO{\dsum \limits_{i=2}^{n}}%
%BeginExpansion
{\displaystyle \sum \limits_{i=2}^{n}}
%EndExpansion
x_{i}y_{i},
\]
where $X=\left(  x_{1},x_{2},...,x_{n}\right)  ,Y=\left(  y_{1},y_{2}%
,...,y_{n}\right)  $ is the usual coordinate system. An arbitrary vector
$X=\left(  x_{1},x_{2},...,x_{n}\right)  $\ in $E_{1}^{n\text{ }}$\ can have
one of three Lorentzian causal characters; it can be spacelike if
$\left \langle X,X\right \rangle >0$ or\ $X=0$\ , timelike if $\left \langle
X,X\right \rangle <0$ and null (lightlike) if $\left \langle X,X\right \rangle
=0$\ and $X\neq0$. The category into which a given tangent vector falls is
called its causal character. These definitions can be generalized for curves
as follows. A curve $\alpha$\ in $E_{1}^{n\text{ }}$ is said to be spacelike
if all of its velocity vectors $\alpha^{\prime}$\ are spacelike, similarly for
timelike and null \cite{barros}.

Let $\alpha:I\subset R\mathbb{\longrightarrow}E_{1}^{n\text{ }}$ be non-null
curve in $E_{1}^{n\text{ }}$. A non-null curve $\alpha(s)$ is said to be a
unit speed curve if $\left \langle \alpha^{\prime}(s),\alpha^{\prime
}(s)\right \rangle =\varepsilon_{0}$, ($\varepsilon_{0}$\ being $+1 $ or $-1$
according to $\alpha$\ is spacelike or timelike respectively). Let $\left \{
V_{1},V_{2},...,V_{n}\right \}  $ be the moving Frenet frame along the unit
speed curve $\alpha$, where $V_{i}\left(  i=1,2,...,n\right)  $ denote
$i^{th}$ Frenet vector fields and $k_{i}\left(  i=1,2,...,n-1\right)  $
denotes the $i^{th}$ curvature function of the curve. Then the Frenet formulas
are given as%
\begin{align*}
V_{1}^{\prime}  & =k_{1}V_{2},\\
V_{i}^{\prime}  & =-\varepsilon_{i-2}\varepsilon_{i-1}k_{i-1}V_{i-1}%
+k_{i}V_{i+1},\text{ \ }1<i<n,\\
V_{n}^{\prime}  & =-\varepsilon_{n-2}\varepsilon_{n-1}k_{n-1}V_{n-1}%
+k_{i}V_{i+1},
\end{align*}
where $\left \langle V_{i},V_{i}\right \rangle =\varepsilon_{i-1}=\mp1$
\cite{ilarslan}$.$

\section{Inextinsible Flows of Curve in $E_{1}^{n}$}

\bigskip Unless otherwise stated we assume that.%
\[
\alpha:\left[  0,l\right]  \times \left[  0,w\right)  \mathbb{\longrightarrow
}E_{1}^{n\text{ }}%
\]
is a one parameter family of smooth non-null curves in $E_{1}^{n\text{ }}$,
where $l$\ is the arclength of the initial curve. Suppose that $u$\ is the
curve parametrization variable , $0\leq u\leq l$. If the speed non-null curve
$\alpha$\ is given by $v=\left \Vert \frac{d\alpha}{du}\right \Vert , $ then the
arclength of $\alpha$\ is given as a function of $u$ by%
\[
s(u)=%
%TCIMACRO{\dint \limits_{0}^{u}}%
%BeginExpansion
{\displaystyle \int \limits_{0}^{u}}
%EndExpansion
\left \Vert \frac{\partial \alpha}{\partial u}\right \Vert du=%
%TCIMACRO{\dint \limits_{0}^{u}}%
%BeginExpansion
{\displaystyle \int \limits_{0}^{u}}
%EndExpansion
vdu.
\]
The operator $\frac{\partial}{\partial s}$\ is given by
\begin{equation}
\frac{\partial}{\partial s}=\frac{1}{v}\frac{\partial}{\partial u}.\label{3.1}%
\end{equation}
In this case; the arclength is as follows $ds=vdu$.

\begin{definition}
Let $\alpha$ be a differentiable non-null curve and $\left \{  V_{1}%
,V_{2},...,V_{n}\right \}  $\ be the Frenet frame of $\alpha$\ in Euclidean
n-space. Any flow of the non-null curve can be expressed as follows%
\[
\frac{\partial \alpha}{\partial t}=%
%TCIMACRO{\dsum \limits_{i=1}^{n}}%
%BeginExpansion
{\displaystyle \sum \limits_{i=1}^{n}}
%EndExpansion
f_{i}V_{i}.
\]
Here, $f_{i}$ is the $i^{th}$ scalar speed of the non-null\ curve $\alpha.$
\end{definition}

Let the arclength variation be%
\[
s(u,t)=%
%TCIMACRO{\dint \limits_{0}^{u}}%
%BeginExpansion
{\displaystyle \int \limits_{0}^{u}}
%EndExpansion
vdu.
\]
In $E_{1}^{n\text{ }},$the requirement that the non-null curve not be subject
to any elongation or compression can be expressed by the condition
\begin{equation}
\frac{\partial}{\partial t}s(u,t)=%
%TCIMACRO{\dint \limits_{0}^{u}}%
%BeginExpansion
{\displaystyle \int \limits_{0}^{u}}
%EndExpansion
\frac{\partial v}{\partial t}du=0,\text{ \ }u\in \left[  0,l\right]
.\label{3.2}%
\end{equation}
where $u\in \left[  0,l\right]  .$

\begin{definition}
Let $\alpha$\ be a non-null curve in $E_{1}^{n\text{ }}.$ A non-null curve
evolution $\alpha(u,t)$\ and its flow $\frac{\partial \alpha}{\partial t}$ are
said to be inextensible if%
\[
\frac{\partial}{\partial t}\left \Vert \frac{\partial \alpha}{\partial
u}\right \Vert =0.
\]
Before deriving the necessary and sufficient condition for inelastic
non-null\ curve flow, we need the following lemma.

\begin{lemma}
Let $\left \{  V_{1},V_{2},...,V_{n}\right \}  $\ be the Frenet frame of
non-null curve $\alpha \ $and $\frac{\partial \alpha}{\partial t}=%
%TCIMACRO{\dsum \limits_{i=1}^{n}}%
%BeginExpansion
{\displaystyle \sum \limits_{i=1}^{n}}
%EndExpansion
f_{i}V_{i}$ be a smooth flow of $\alpha$ in $E_{1}^{n\text{ }}.$ Then we have
the following equality:%
\begin{equation}
\frac{\partial v}{\partial t}=\varepsilon_{0}\frac{\partial f_{1}}{\partial
u}-\varepsilon_{1}f_{2}vk_{1}.\label{3.3}%
\end{equation}

\begin{proof}
As $\frac{\partial}{\partial u}$\ and $\frac{\partial}{\partial t}$\ commute
and $v^{2}=\left \langle \frac{\partial \alpha}{\partial u},\frac{\partial
\alpha}{\partial u}\right \rangle ,$\ we have%
\begin{align*}
2v\frac{\partial v}{\partial t}  & =\frac{\partial}{\partial t}\left \langle
\frac{\partial \alpha}{\partial u},\frac{\partial \alpha}{\partial
u}\right \rangle \\
& =2\left \langle \frac{\partial \alpha}{\partial u},\frac{\partial}{\partial
u}\left(
%TCIMACRO{\dsum \limits_{i=1}^{n}}%
%BeginExpansion
{\displaystyle \sum \limits_{i=1}^{n}}
%EndExpansion
f_{i}V_{i}\right)  \right \rangle \\
& =2\left \langle vV_{1},%
%TCIMACRO{\dsum \limits_{i=1}^{n}}%
%BeginExpansion
{\displaystyle \sum \limits_{i=1}^{n}}
%EndExpansion
\frac{\partial f_{i}}{\partial u}V_{i}+%
%TCIMACRO{\dsum \limits_{i=1}^{n}}%
%BeginExpansion
{\displaystyle \sum \limits_{i=1}^{n}}
%EndExpansion
f_{i}\frac{\partial V_{i}}{\partial u}\right \rangle \\
& =2\left \langle vV_{1},\frac{\partial f_{1}}{\partial u}V_{1}+f_{1}%
\frac{\partial V_{1}}{\partial u}+...+\frac{\partial f_{n}}{\partial u}%
V_{n}+f_{n}\frac{\partial V_{n}}{\partial u}\right \rangle \\
& =2\left \langle vV_{1},\frac{\partial f_{1}}{\partial u}V_{1}+f_{1}%
vk_{1}V_{2}+...+\frac{\partial f_{n}}{\partial u}V_{n}-f_{n}\varepsilon
_{n-2}\varepsilon_{n-1}vk_{n-1}V_{n-1}\right \rangle \\
& =2\left(  \varepsilon_{0}\frac{\partial f_{1}}{\partial u}-\varepsilon
_{1}f_{2}vk_{1}\right)  .
\end{align*}
This clearly forces%
\[
\frac{\partial v}{\partial t}=\varepsilon_{0}\frac{\partial f_{1}}{\partial
u}-\varepsilon_{1}f_{2}vk_{1}.
\]

\end{proof}
\end{lemma}
\end{definition}

\begin{theorem}
\label{teo3.1}\bigskip Let $\left \{  V_{1},V_{2},...,V_{n}\right \}  $ be the
moving Frenet frame of the non-null curve $\alpha$ and $\frac{\partial \alpha
}{\partial t}=%
%TCIMACRO{\dsum \limits_{i=1}^{n}}%
%BeginExpansion
{\displaystyle \sum \limits_{i=1}^{n}}
%EndExpansion
f_{i}V_{i}$\ be a differentiable flow of $\alpha$\ in $E_{1}^{n\text{ }}$. In
this case, the flow is inextensible if and only if
\begin{equation}
\frac{\partial f_{1}}{\partial s}=\varepsilon_{0}\varepsilon_{1}f_{2}%
k_{1}.\label{3.4}%
\end{equation}

\end{theorem}

\begin{proof}
Let us assume that the non-null curve flow is inextensible. From equations
(\ref{3.2}) and (\ref{3.3}) it follows that%
\[
\frac{\partial}{\partial t}s(u,t)=%
%TCIMACRO{\dint \limits_{0}^{u}}%
%BeginExpansion
{\displaystyle \int \limits_{0}^{u}}
%EndExpansion
\frac{\partial v}{\partial t}du=%
%TCIMACRO{\dint \limits_{0}^{u}}%
%BeginExpansion
{\displaystyle \int \limits_{0}^{u}}
%EndExpansion
\left(  \varepsilon_{0}\frac{\partial f_{1}}{\partial u}-\varepsilon_{1}%
f_{2}vk_{1}\right)  du=0,\text{ \ }u\in \left[  0,l\right]  .
\]
This clearly forces%
\[
\varepsilon_{0}\frac{\partial f_{1}}{\partial u}-\varepsilon_{1}f_{2}vk_{1}=0.
\]
Combining the last equation with (\ref{3.1}) yields%
\[
\frac{\partial f_{1}}{\partial s}=\varepsilon_{0}\varepsilon_{1}f_{2}k_{1}.
\]
On the contrary, following similar way as above, the proof can be completed.

Now, suppose that the non-null curve $\alpha$ is a arclength parametrized
curve. That is, $v=1$\ and the local coordinate $u$\ corresponds to the curve
arclength $s$.

\begin{lemma}
Let $\left \{  V_{1},V_{2},...,V_{n}\right \}  $\ be the moving Frenet frame of
the non-null curve $\alpha$. The differentions of $\left \{  V_{1}%
,V_{2},...,V_{n}\right \}  $\ with respect to $t$\ is%
\begin{align*}
\frac{\partial V_{1}}{\partial t}  & =\left[
%TCIMACRO{\dsum \limits_{i=2}^{n-1}}%
%BeginExpansion
{\displaystyle \sum \limits_{i=2}^{n-1}}
%EndExpansion
\left(  f_{i-1}k_{i-1}+\frac{\partial f_{i}}{\partial s}-\varepsilon
_{i-1}\varepsilon_{i}f_{i+1}k_{i}\right)  V_{i}\right]  +\left(
f_{n-1}k_{n-1}+\frac{\partial f_{n}}{\partial s}\right)  V_{n},\\
\frac{\partial V_{j}}{\partial t}  & =-\varepsilon_{0}\left(  \varepsilon
_{j-1}f_{j-1}k_{j-1}+\varepsilon_{j-1}\frac{\partial f_{j}}{\partial
s}-\varepsilon_{j}f_{j+1}k_{j}\right)  V_{1}+\left[
%TCIMACRO{\dsum \limits_{\substack{k=2 \\k\neq j}}^{n}}%
%BeginExpansion
{\displaystyle \sum \limits_{\substack{k=2 \\k\neq j}}^{n}}
%EndExpansion
\Psi_{kj}V_{k}\right]  ,\text{ \ }1<j<n,\\
\frac{\partial V_{n}}{\partial t}  & =-\varepsilon_{0}\varepsilon_{n-1}\left(
f_{n-1}k_{n-1}+\frac{\partial f_{n}}{\partial s}\right)  V_{1}+\left[
%TCIMACRO{\dsum \limits_{k=2}^{n-1}}%
%BeginExpansion
{\displaystyle \sum \limits_{k=2}^{n-1}}
%EndExpansion
\Psi_{kn}V_{k}\right]  ,
\end{align*}
where $\Psi_{kj}=\left \langle \frac{\partial V_{j}}{\partial t},V_{k}%
\right \rangle ,$ $k$ $\neq j,$ $1\leq j,k\leq n$\ and $\varepsilon
_{i-1}=\left \langle V_{i},V_{i}\right \rangle =\pm1,$ $1\leq i\leq n$.

\begin{proof}
As $\frac{\partial}{\partial t}$\ and $\frac{\partial}{\partial s}$\ commute,
we have
\begin{align*}
\frac{\partial V_{1}}{\partial t}  & =\frac{\partial}{\partial t}\left(
\frac{\partial \alpha}{\partial s}\right)  =\frac{\partial}{\partial s}\left(
\frac{\partial \alpha}{\partial t}\right)  =\frac{\partial}{\partial s}\left(
%TCIMACRO{\dsum \limits_{i=1}^{n}}%
%BeginExpansion
{\displaystyle \sum \limits_{i=1}^{n}}
%EndExpansion
f_{i}V_{i}\right)  =%
%TCIMACRO{\dsum \limits_{i=1}^{n}}%
%BeginExpansion
{\displaystyle \sum \limits_{i=1}^{n}}
%EndExpansion
\frac{\partial f_{i}}{\partial s}V_{i}+%
%TCIMACRO{\dsum \limits_{i=1}^{n}}%
%BeginExpansion
{\displaystyle \sum \limits_{i=1}^{n}}
%EndExpansion
f_{i}\frac{\partial V_{i}}{\partial s}\\
& =\frac{\partial f_{1}}{\partial s}V_{1}+f_{1}\frac{\partial V_{1}}{\partial
s}+\frac{\partial f_{2}}{\partial s}V_{2}+f_{2}\frac{\partial V_{2}}{\partial
s}+...+\frac{\partial f_{n}}{\partial s}V_{n}+f_{n}\frac{\partial V_{n}%
}{\partial s}\\
& =\frac{\partial f_{1}}{\partial s}V_{1}+f_{1}k_{1}V_{2}+\frac{\partial
f_{2}}{\partial s}V_{2}+f_{2}\left(  -\varepsilon_{0}\varepsilon_{1}k_{1}%
V_{1}+k_{2}V_{3}\right)  +...+\frac{\partial f_{n}}{\partial s}V_{n}%
-f_{n}\varepsilon_{n-2}\varepsilon_{n-1}k_{n-1}V_{n-1}.
\end{align*}
Substituting the equation (\ref{3.4}) into the last equation yields%
\[
\frac{\partial V_{1}}{\partial t}=\left[
%TCIMACRO{\dsum \limits_{i=2}^{n-1}}%
%BeginExpansion
{\displaystyle \sum \limits_{i=2}^{n-1}}
%EndExpansion
\left(  f_{i-1}k_{i-1}+\frac{\partial f_{i}}{\partial s}-\varepsilon
_{i-1}\varepsilon_{i}f_{i+1}k_{i}\right)  V_{i}\right]  +\left(
f_{n-1}k_{n-1}+\frac{\partial f_{n}}{\partial s}\right)  V_{n}.
\]
Now, differentiating the Frenet frame with respect to $t$ for $1<j<n$\
\begin{align}
0  & =\frac{\partial}{\partial t}\left \langle V_{1},V_{j}\right \rangle
=\left \langle \frac{\partial V_{1}}{\partial t},V_{j}\right \rangle
+\left \langle V_{1},\frac{\partial V_{j}}{\partial t}\right \rangle \nonumber \\
& =\left(  \varepsilon_{j-1}f_{j-1}k_{j-1}+\varepsilon_{j-1}\frac{\partial
f_{j}}{\partial s}-\varepsilon_{j}f_{j+1}k_{j}\right)  +\left \langle
V_{1},\frac{\partial V_{j}}{\partial t}\right \rangle .\label{3.5}%
\end{align}
Thus, from the last equation we get
\[
\frac{\partial V_{j}}{\partial t}=-\varepsilon_{0}\left(  \varepsilon
_{j-1}f_{j-1}k_{j-1}+\varepsilon_{j-1}\frac{\partial f_{j}}{\partial
s}-\varepsilon_{j}f_{j+1}k_{j}\right)  V_{1}+\left[
%TCIMACRO{\dsum \limits_{\substack{k=2 \\k\neq j}}^{n}}%
%BeginExpansion
{\displaystyle \sum \limits_{\substack{k=2 \\k\neq j}}^{n}}
%EndExpansion
\Psi_{kj}V_{k}\right]  .
\]
Since $\left \langle V_{1},V_{n}\right \rangle =0,$\ this follows by the similar
method as above%
\[
\frac{\partial V_{n}}{\partial t}=-\varepsilon_{0}\varepsilon_{n-1}\left(
f_{n-1}k_{n-1}+\frac{\partial f_{n}}{\partial s}\right)  V_{1}+\left[
%TCIMACRO{\dsum \limits_{k=2}^{n-1}}%
%BeginExpansion
{\displaystyle \sum \limits_{k=2}^{n-1}}
%EndExpansion
\Psi_{kn}V_{k}\right]  .
\]

\begin{theorem}
Let the non-null curve flow $\frac{\partial \alpha}{\partial t}=%
%TCIMACRO{\dsum \limits_{i=1}^{n}}%
%BeginExpansion
{\displaystyle \sum \limits_{i=1}^{n}}
%EndExpansion
f_{i}V_{i}$\ be inextensible in $E_{1}^{n\text{ }}$. Then, there exist the
following system of partial differential equations.%
\begin{align*}
\frac{\partial k_{1}}{\partial t}  & =\varepsilon_{0}\varepsilon_{1}f_{2}%
k_{1}^{2}+f_{1}\frac{\partial k_{1}}{\partial s}+\frac{\partial^{2}f_{2}%
}{\partial s^{2}}-2\varepsilon_{1}\varepsilon_{2}\frac{\partial f_{3}%
}{\partial s}k_{2}-\varepsilon_{1}\varepsilon_{2}f_{3}\frac{\partial k_{2}%
}{\partial s}-\varepsilon_{1}\varepsilon_{2}f_{2}k_{2}^{2}-\varepsilon
_{1}\varepsilon_{3}f_{4}k_{2}k_{3},\\
\frac{\partial k_{i-1}}{\partial t}  & =-\varepsilon_{i-2}\varepsilon
_{i-1}\frac{\partial \Psi_{(i-1)i}}{\partial s}-\varepsilon_{i-2}%
\varepsilon_{i-1}\Psi_{(i-2)i}k_{i-2},\\
\frac{\partial k_{i}}{\partial t}  & =\frac{\partial \Psi_{(i+1)i}}{\partial
s}-\varepsilon_{i}\varepsilon_{i+1}\Psi_{(i+2)i}k_{i+1},\\
\frac{\partial k_{n-1}}{\partial t}  & =-\varepsilon_{n-2}\varepsilon
_{n-1}\frac{\partial \Psi_{(n-1)n}}{\partial s}-\varepsilon_{n-2}%
\varepsilon_{n-1}\Psi_{(n-2)n}k_{n-2}.
\end{align*}

\end{theorem}
\end{proof}
\end{lemma}
\end{proof}

\begin{proof}
\bigskip Noting that $\frac{\partial}{\partial s}\frac{\partial V_{1}%
}{\partial t}=\frac{\partial}{\partial t}\frac{\partial V_{1}}{\partial s}$
thus we have%
\begin{align}
\frac{\partial}{\partial s}\frac{\partial V_{1}}{\partial t}  & =\frac
{\partial}{\partial s}\left[
%TCIMACRO{\dsum \limits_{i=2}^{n-1}}%
%BeginExpansion
{\displaystyle \sum \limits_{i=2}^{n-1}}
%EndExpansion
\left(  f_{i-1}k_{i-1}+\frac{\partial f_{i}}{\partial s}-\varepsilon
_{i-1}\varepsilon_{i}f_{i+1}k_{i}\right)  V_{i}+\left(  f_{n-1}k_{n-1}%
+\frac{\partial f_{n}}{\partial s}\right)  V_{n}\right] \nonumber \\
& =%
%TCIMACRO{\dsum \limits_{i=2}^{n-1}}%
%BeginExpansion
{\displaystyle \sum \limits_{i=2}^{n-1}}
%EndExpansion
\left[  \left(  \frac{\partial f_{i-1}}{\partial s}k_{i-1}+f_{i-1}%
\frac{\partial k_{i-1}}{\partial s}+\frac{\partial^{2}f_{i}}{\partial s^{2}%
}-\varepsilon_{i-1}\varepsilon_{i}\frac{\partial f_{i+1}}{\partial s}%
k_{i}-\varepsilon_{i-1}\varepsilon_{i}f_{i+1}\frac{\partial k_{i}}{\partial
s}\right)  V_{i}\right] \nonumber \\
& +%
%TCIMACRO{\dsum \limits_{i=2}^{n-1}}%
%BeginExpansion
{\displaystyle \sum \limits_{i=2}^{n-1}}
%EndExpansion
\left[  \left(  f_{i-1}k_{i-1}+\frac{\partial f_{i}}{\partial s}%
-\varepsilon_{i-1}\varepsilon_{i}f_{i+1}k_{i}\right)  \frac{\partial V_{i}%
}{\partial s}\right] \label{3.8}\\
& +\left(  \frac{\partial f_{n-1}}{\partial s}k_{n-1}+f_{n-1}\frac{\partial
k_{n-1}}{\partial s}+\frac{\partial^{2}f_{n}}{\partial s^{2}}\right)
V_{n}+\left(  f_{n-1}k_{n-1}+\frac{\partial f_{n}}{\partial s}\right)
\frac{\partial V_{n}}{\partial s}\nonumber
\end{align}
while
\begin{equation}
\frac{\partial}{\partial t}\frac{\partial V_{1}}{\partial s}=\frac{\partial
}{\partial t}\left(  k_{1}V_{2}\right)  =\frac{\partial k_{1}}{\partial
t}V_{2}+k_{1}\frac{\partial V_{2}}{\partial t}.\label{3.9}%
\end{equation}
Therefore, from the equation (\ref{3.8}) and (\ref{3.9}) it is seen that%
\[
\frac{\partial k_{1}}{\partial t}=\varepsilon_{0}\varepsilon_{1}f_{2}k_{1}%
^{2}+f_{1}\frac{\partial k_{1}}{\partial s}+\frac{\partial^{2}f_{2}}{\partial
s^{2}}-2\varepsilon_{1}\varepsilon_{2}\frac{\partial f_{3}}{\partial s}%
k_{2}-\varepsilon_{1}\varepsilon_{2}f_{3}\frac{\partial k_{2}}{\partial
s}-\varepsilon_{1}\varepsilon_{2}f_{2}k_{2}^{2}-\varepsilon_{1}\varepsilon
_{3}f_{4}k_{2}k_{3}.
\]
Since $\frac{\partial}{\partial s}\frac{\partial V_{i}}{\partial t}%
=\frac{\partial}{\partial t}\frac{\partial V_{i}}{\partial s}$, we obtain%
\begin{align*}
\frac{\partial}{\partial s}\frac{\partial V_{i}}{\partial t}  & =\frac
{\partial}{\partial s}\left[  -\varepsilon_{0}\left(  \varepsilon_{i-1}%
f_{i-1}k_{i-1}+\varepsilon_{i-1}\frac{\partial f_{i}}{\partial s}%
-\varepsilon_{i}f_{i+1}k_{i}\right)  V_{1}+\left(
%TCIMACRO{\dsum \limits_{\substack{k=2 \\k\neq i}}^{n}}%
%BeginExpansion
{\displaystyle \sum \limits_{\substack{k=2 \\k\neq i}}^{n}}
%EndExpansion
\Psi_{ki}V_{k}\right)  \right] \\
& =\varepsilon_{0}\left(  -\varepsilon_{i-1}\frac{\partial f_{i-1}}{\partial
s}k_{i-1}-\varepsilon_{i-1}f_{i-1}\frac{\partial k_{i-1}}{\partial
s}-\varepsilon_{i-1}\frac{\partial^{2}f_{i}}{\partial s^{2}}+\varepsilon
_{i}\frac{\partial f_{i+1}}{\partial s}k_{i}+\varepsilon_{i}f_{i+1}%
\frac{\partial k_{i}}{\partial s}\right)  V_{1}\\
& +\left(  -\varepsilon_{0}\varepsilon_{i-1}f_{i-1}k_{i-1}-\varepsilon
_{0}\varepsilon_{i-1}\frac{\partial f_{i}}{\partial s}+\varepsilon
_{0}\varepsilon_{i}f_{i+1}k_{i}\right)  \frac{\partial V_{1}}{\partial s}+%
%TCIMACRO{\dsum \limits_{\substack{k=2 \\k\neq i}}^{n}}%
%BeginExpansion
{\displaystyle \sum \limits_{\substack{k=2 \\k\neq i}}^{n}}
%EndExpansion
\left(  \frac{\partial \Psi_{ki}}{\partial s}V_{k}+\Psi_{ki}\frac{\partial
V_{k}}{\partial s}\right)
\end{align*}
while%
\begin{align*}
\frac{\partial}{\partial t}\frac{\partial V_{i}}{\partial s}  & =\frac
{\partial}{\partial t}\left(  -\varepsilon_{i-2}\varepsilon_{i-1}%
k_{i-1}V_{i-1}+k_{i}V_{i+1}\right) \\
& =-\varepsilon_{i-2}\varepsilon_{i-1}\frac{\partial k_{i-1}}{\partial
t}V_{i-1}-\varepsilon_{i-2}\varepsilon_{i-1}k_{i-1}\frac{\partial V_{i-1}%
}{\partial t}+\frac{\partial k_{i}}{\partial t}V_{i+1}+k_{i}\frac{\partial
V_{i+1}}{\partial t}.
\end{align*}
Hence
\[
\frac{\partial k_{i-1}}{\partial t}=-\varepsilon_{i-2}\varepsilon_{i-1}%
\frac{\partial \Psi_{(i-1)i}}{\partial s}-\varepsilon_{i-2}\varepsilon
_{i-1}\Psi_{(i-2)i}k_{i-2}%
\]
and%
\[
\frac{\partial k_{i}}{\partial t}=\frac{\partial \Psi_{(i+1)i}}{\partial
s}-\varepsilon_{i}\varepsilon_{i+1}\Psi_{(i+2)i}k_{i+1}.
\]
By same way as above and considering $\frac{\partial}{\partial s}%
\frac{\partial V_{n}}{\partial t}=\frac{\partial}{\partial t}\frac{\partial
V_{n}}{\partial s}$\ we reach\bigskip%
\[
\frac{\partial k_{n-1}}{\partial t}=-\varepsilon_{n-2}\varepsilon_{n-1}%
\frac{\partial \Psi_{(n-1)n}}{\partial s}-\varepsilon_{n-2}\varepsilon
_{n-1}\Psi_{(n-2)n}k_{n-2}.
\]

\end{proof}

\end{document}